\documentclass[twoside,11pt,reqno]{amsart}
\usepackage{amsmath,amssymb,amscd,mathrsfs,amscd, todonotes}
\usepackage{graphics,verbatim}
\usepackage{todonotes}
\usepackage{hyperref}

\let\oldsection=\section

\newcommand{\losemi}{{\otimes \kern -.78em \ltimes}}
\newcommand{\rosemi}{{\otimes \kern -.78em \rtimes}}

\newcommand{\Hom}{\ensuremath{\operatorname{Hom}}}

\newcommand{\sgn}{\operatorname{sgn}}
\newcommand{\Ext}{\operatorname{Ext}}

\newcommand{\C}{\mathbb{C}}

\newcommand{\res}{\ensuremath{\operatorname{res}}}

\newcommand{\tr}{\ensuremath{\operatorname{tr} }}

\newcommand{\HH}{\operatorname{H}}

\renewcommand{\mod}{\operatorname{mod}}

\newcommand{\He}{\mathcal{H}}

\makeatletter
\newcommand{\leqnomode}{\tagsleft@true}
\newcommand{\reqnomode}{\tagsleft@false}
\makeatother

\newtheorem{theorem}{Theorem}[subsection]

\makeatletter\let\c@fact\c@theorem\makeatother

\makeatletter\let\c@note\c@theorem\makeatother

\newtheorem{lemma}{Lemma}[subsection]
\makeatletter\let\c@lemma\c@theorem\makeatother

\makeatletter\let\c@lemma\c@theorem\makeatother

\makeatletter\let\c@alg\c@theorem\makeatother

\newtheorem{prop}{Proposition}[subsection]
\makeatletter\let\c@prop\c@theorem\makeatother

\makeatletter\let\c@conj\c@theorem\makeatother

\newtheorem{cor}{Corollary}[subsection]
\makeatletter\let\c@cor\c@theorem\makeatother

\newtheorem{defn}{Definition}[subsection]
\makeatletter\let\c@defn\c@theorem\makeatother

\theoremstyle{definition}

\newtheorem{remark}{Remark}[subsection]
\makeatletter\let\c@remark\c@theorem\makeatother

\makeatletter\let\c@example\c@theorem\makeatother
\numberwithin{equation}{subsection}

\usepackage[capitalise]{cleveref}
\crefname{theorem}{Theorem}{Theorems}
\crefname{fact}{Fact}{Facts}
\crefname{note}{Note}{Notes}
\crefname{lemma}{Lemma}{Lemmas}
\crefname{alg}{Algorithm}{Algorithms}
\crefname{remark}{Remark}{Remarks}
\crefname{example}{Example}{Examples}
\crefname{prop}{Proposition}{Propositions}
\crefname{conj}{Conjecture}{Conjectures}
\crefname{cor}{Corollary}{Corollaries}
\crefname{defn}{Definition}{Definitions}
\crefname{equation}{\!\!}{\!\!} %Remove spacing around phantom equation name

\newcounter{listequation}

\begin{document}

\title[Support varieties]
{\bf Support varieties for Hecke algebras}
\author{\sc Daniel K. Nakano}
\address
{Department of Mathematics\\ University of Georgia \\
Athens\\ GA~30602, USA}
\thanks{Research of the first author was supported in part by NSF
grant  DMS-1701768}
\email{nakano@math.uga.edu}
\author{\sc Ziqing Xiang}
\address
{Department of Mathematics\\ University of Georgia \\
Athens\\ GA~30602, USA}
\email{ziqing@uga.edu}
\date{February 2018}
\begin{abstract}
Let $\He_{q}(d)$ be the Iwahori-Hecke algebra for the symmetric group, where $q$ is a primitive $l$th root of unity. In this 
paper we develop a theory of support varieties which detects natural homological properties such as the complexity of modules. The theory 
the authors develop has a canonical description in an affine space where computations are tractable. The ideas involve the interplay with 
the computation of the cohomology ring due to Benson, Erdmann and Mikaelian, the theory of vertices due to Dipper and Du, and 
branching results for cohomology by Hemmer and Nakano. Calculations of support varieties and vertices are presented for permutation, Young and 
classes of Specht modules. Furthermore, a discussion of how the authors' results can be extended to other Hecke algebras for other classical groups is presented at the end of the paper.  
\end{abstract}
\maketitle

\maketitle
\vskip 1cm
%%%%%%%%%%%%%%%%%%%
%Introduction
%%%%%%%%%%%%%%%%%%%

\section{Introduction}

\subsection{} Support varieties have been developed  in a variety of contexts that involve categories which are Frobenius 
(i.e., where injectivity and projectivity are equivalent) and have a monoidal tensor structure. The monoidal tensor structure 
generally arises from a Hopf structure on an underlying algebra.  Examples of such categories include modules for finite group schemes, quantum groups and 
Lie superalgebras (cf. \cite{FP1,FP2} \cite{FPe}, \cite{NPal}, \cite{BNPP}). More recently, the key properties of support varieties have be used to create axiomatic support theory 
and tensor triangular geometry. Very little is known about extracting geometric properties from Frobenius categories where there is no underlying coproduct.  

In this paper we will develop a support variety theory for the Iwahori-Hecke algebra for Weyl group for the symmetric group (i.e., type $A$), and for Hecke algebras for other classical groups. 
In general, the module category for Hecke algebras lacks a tensor structure. This presents major difficulties in executing important constructions. Our modified theory of support varieties differs from approaches proposed using the 
Hochschild cohomology (cf. \cite{L}). In those contexts, varieties can be defined, however it is not clear how (i) these varieties can be computed and (ii) how they can be used 
in the general theory. It is anticipated that our methods along with several recent developments in extending the theory in type $A$ to other Weyl groups (cf. \cite{DPS2,DPS3}) 
might lead to a general finite generation results entailing the cohomology ring and the creation of a general theory of supports with realizations for arbitrary Hecke algebras. 

\subsection{} The paper is organized as follows. In Section 2, we introduce the conventions and notation that will be used throughout the paper. The following 
section, Section 3, provides the definition and details of the results on transfer and its relationship to cohomology. In Section 4, using the explicit description of the cohomology ring $R_d :=
\HH^{\bullet}(\He_{q}(d), \C)$ due to Benson, Erdmann and Mikaelian \cite{BEM} we show that (i) $R_{\lambda}:=\HH^{\bullet}(\He_{q}(\lambda), \C)$
is finitely-generated and (ii) $\Ext^{\bullet}_{\He_{q}(\lambda)}(\C, M)$ is finitely generated as a $R_{\lambda}$-module for any composition $\lambda$. Here 
$M$ is a finite-dimensional $\He_{q}(\lambda)$-module. The results above allow one to use the ideas involving branching to Young subgroups 
from Hemmer and Nakano \cite{HN} to construct support varieties for any $\He_{q}(d)$. These ideas were important for the recent proof of the Erdmann-Lim-Tan Conjecture \cite{ELT}
by Cohen, Hemmer and Nakano \cite{CHN} that involved computing the complexity of the Lie module. Our results rely heavily on the work of Dipper and Du (cf. \cite{DD,Du}) that provides the technical machinery to prove many of the results in this section. 

Following the seminal work of Alperin, we define complexity for $\He_{q}(d)$-modules in Section 5. The main point of this section is to utilize the 
representation theory of the Hecke algebras to demonstrate that the complexity of a module is in fact equal to the dimension of our support varieties (as defined in Section 4). As an application we prove that the complexity of any modules is 
less than the complexity of the trivial module. Note that without a tensor structure (as in our case) this is a non-trivial fact. Subsequently, in Section 6, we compute the complexity and varieties for Young and permutation modules, which extends the 
earlier work in \cite{HN} for symmetric groups to Hecke algebras of type $A$. 

In Section 7, we construct a new invariant for Specht modules called the graded dimension and relate this graded dimension to product of cyclotomic polynomials. This definition in conjunction with results for relative cohomology 
allows us to show that the vertex of the Specht module satisfies certain numerical constraints. As a by-product, we are able to explicitly compute the vertex of Specht modules for a certain class of partitions. Finally, in Section 8, 
we apply our results for Hecke algebras of type A with various Morita equivalences to construct support varieties for Hecke algebras of types $B/C$ and $D$, and show that the complexity for modules for these algebras 
is equal to the dimension of the corresponding varieties. Several open questions of further interest are posed at the end of the paper.

\section{Notation and Preliminaries} \label{notation}

\subsection{} Throughout this paper, we will work over the complex numbers $\C$, although many of the constructions will work over an arbitrary field. Let $q\in \C^{*}$ and $\Sigma_{d}$ be the symmetric group on $d$ letters. The Hecke algebra $\He_q(d) := \He_{q}(\Sigma_d)$ is the free $\C$-module with basis $\{ T_w : w \in \Sigma_d \}$ with multiplication defined by 
\[ T_w T_s := \begin{cases}
	T_{ws}, & \text{if $\ell(ws) > \ell(w)$,} \\
	qT_{ws} + (q-1)T_w, & \text{otherwise,}
\end{cases} \]
where $s = (i, i + 1) \in \Sigma_d$ is a simple transposition and $w \in \Sigma_d$. The function $\ell : \Sigma_{d} \to \mathbb{N}$ is the usual length function that is defined for any Weyl group. 

There is an automorphism $^\#$ and an antiautomorphism $^\ast$ of $\He_q(d)$ defined by:
\[ T_w^\# := (-q)^{\ell(w)}(T_{w^{-1}})^{-1}, \text{\quad and } T_w^\ast := T_{w^{-1}}. \]
The maps $^\#$ and $^\ast$ are both involutions. We will also use the dual $^\vee$ defined by:
\[ T_w^\vee := q^{-\ell(w)}T_{w^{-1}}. \]
For any $\He_q(d)$-module $M$ one can define a dual module $M^\ast := \Hom_{\C}(M, \C)$, where the action of $\He_q(d)$ is given by $(hf)(m) := f(h^\ast m)$.

Let $l$ be the smallest integer such that $1 + q + \cdots + q^{l - 1} = 0$, and set $l := \infty$ if no such integer exists. If $q\in \C^{*}$ is a primitive $j$th root of unity then $l = j$. Furthermore, 
if $q$ is a not root of unity then $\He_q(d)$ is semisimple. 

Let $\Lambda(d) := \{\lambda \vDash d\}$ be the set of all {\em compositions} of $d$ and $\Lambda^+(d) := \{\lambda \vdash d\}$ be the set of all {\em partitions} of $d$. 
Given two compositions $\lambda, \mu \in \Lambda(d)$ (resp. partitions), we denote $\mu \vDash \lambda$ (resp. $\mu \vdash \lambda$) if $\mu$ is finer than $\lambda$. 
A partition/composition $\lambda$ of $d$ is called {\em $l$-parabolic} if every part of $\lambda$ is either $1$ or $l$, and it is {\em simple} provided that exactly one part of $\lambda$ is not $1$ and all other parts are $1$'s.

A partition $\lambda = (\lambda_1, \lambda_2, \dots)$ is called {\em $l$-restricted} if $\lambda_i - \lambda_{i + 1} \leq l - 1$ for all $i$. The set of the $l$-restricted partitions of $d$ will be denoted by $\Lambda^+_{\res}(d)$. 
A partition $\lambda$ is called {\em $l$-regular} if its {\em transpose} $\lambda'$ is $l$-restricted. The set of all $l$-regular partitions of $d$ is denoted by $\Lambda^+_{\operatorname{reg}}(d)$.

\subsection{} We refer the reader to \cite{DJ1} and \cite{Mat} for details about the representation theory of $\He_q(d)$. The major classes of representations parallels those for the modular representation theory of the symmetric group. 
For each $\lambda \in \Lambda^+(d)$ there is a {\em $q$-Specht module} of the Hecke algebra $\He_q(d)$, denoted by $S^\lambda$. If $\lambda \in \Lambda^+_{\operatorname{reg}}(d)$ then $S^\lambda$ has a unique simple quotient 
denoted by $D^\lambda$. One obtains a complete collection of non-isomorphic simple modules $D^\lambda$ for $\lambda \in \Lambda^+_{\operatorname{reg}}(d)$ for $\He_q(d)$-module in this way. 
These simple modules are self-dual and absolutely irreducible.

For a composition $\lambda \in \Lambda(d)$, let $\Sigma_\lambda$ be the corresponding Young subgroup of $\Sigma_d$, that is $\Sigma_\lambda \cong \Sigma_{\lambda_1} \times \Sigma_{\lambda_2} \times \cdots$. 
Associated to this Young subgroup there is a corresponding subalgebra of $\He_q(d)$:
\[ \He_q(\lambda) := \langle T_w \mid w \in \Sigma_\lambda \rangle \cong \He_q(\lambda_1) \times \He_q(\lambda_2) \times \cdots. \]
Set
\[ x_\lambda := \sum_{w \in \Sigma_\lambda} T_w. \]
Define the {\em permutation module} $M^{\lambda} := \He_q(d) x_\lambda$. One also has the isomorphism $M^{\lambda}\cong \operatorname{ind}_{\He_q(\lambda)}^{\He_q(d)}\C$. 
Given $\lambda \in \Lambda^+(d)$, there is a unique indecomposable direct summand of $M^\lambda$ containing $S^\lambda$ that is the {\em Young module} $Y^\lambda$. 
All other summands are Young modules whose partitions are strictly greater than $\lambda$ in the dominance ordering. Furthermore, $Y^\lambda \cong Y^\mu$ if and only if $\lambda = \mu$. 

The Hecke algebra $\He_q(d)$ has two one-dimensional representations \cite[1.14]{Mat}, which we denote $\C$ and $\sgn$, given by:
\begin{equation} \label{onedim}
\C(T_w) := q^{\ell(w)} \quad \text{and} \quad \sgn(T_w) := (-1)^{\ell(w)}.
\end{equation}
When $q = 1$ these specialize to the usual trivial and sign representations of $\C\Sigma_d$. 

In general the tensor product of two $\He_q(d)$-modules is not an $\He_q(d)$-module, since $\He_q(d)$ is not a Hopf algebra. However the automorphism $\#$ lets us define, for each $\He_q(d)$-module $M$, a 
new module $M^{\#}$ with the same underlying vector space and with action given by $h \circ m := h^\# m$. This specializes for $q = 1$ to tensoring with the sign representation. This is denoted by 
\[ M \otimes \sgn := M^\#. \]
The simple $\He_q(d)$-modules can also be indexed by $\Lambda^{+}_{\res}(d)$. For $\lambda \in \Lambda^{+}_{\res}(d)$ denote the corresponding simple module by $D_\lambda$. 
It is a fact that $D_\lambda = \operatorname{soc}_{\He_{q}(d)}(S^\lambda)$. The relationship between these two labellings is given by:
\begin{equation}
D^\lambda \cong D_{\lambda'} \otimes \sgn \quad \text{for any $\lambda \in \Lambda^{+}_{\operatorname{reg}}(d)$}.
\end{equation}
We remark that that tensoring with the sign representation turns Specht modules into dual Specht modules and vice-versa (cf. \cite[6.7]{J1}, \cite[Exer. 3.14]{Mat}): 
\begin{equation} \label{trdual}
S^\lambda \otimes \sgn  \cong (S^{\lambda'})^\ast := S_{\lambda'}.
\end{equation}

\section{Transfer and Cohomology}

\subsection{} We begin by defining the notion of transfer for Hecke algebras and state one of the main results that involves the decomposition of induced modules. For an algebra $A$ and an $A$-$A$-bimodule, $M$, let $M^A := \{m \in M : a m = m a, \text{for all $a \in A$}\}$.

Let $\lambda$ be a composition and $\mu, \nu \vDash \lambda$ be two subcompositions. We abuse the notation by letting $\Sigma_\lambda / \Sigma_\mu$, $\Sigma_\nu \backslash \Sigma_\lambda$ and $\Sigma_\nu \backslash \Sigma_\lambda / \Sigma_\mu$ 
denote the set of {\em distinguished left/right/double coset representatives}, respectively. For a distinguished double coset representative $w \in \Sigma_\nu \backslash \Sigma_\lambda / \Sigma_\mu$, $\Sigma_\nu \cap \Sigma_\mu^w$ is always a Young subgroup. More precisely, there exists a unique subcomposition of $\lambda$, denoted by $\nu \cap \mu^w \vDash \lambda$, such that $\Sigma_{\nu \cap \mu^w} = \Sigma_\nu \cap \Sigma_\mu^w$, where $\Sigma_\mu^w := w \Sigma_\mu w^{-1}$. One has a 
{\em restriction map} $\res_{\lambda, \mu} : M^{\He_q(\lambda)} \hookrightarrow M^{\He_q(\mu)}$.

\begin{defn}[\cite{Jo, L}] Let $\mu \vDash \lambda$ be two compositions.
For a $\He_q(\lambda)$-$\He_q(\lambda)$-bimodule $M$, the {\em transfer map} $\tr_{\mu, \lambda} : M^{\He_q(\mu)} \to M^{\He_q(\lambda)}$ is defined as follow:
\[ \tr_{\mu, \lambda}(m) := \sum_{w \in \Sigma_\lambda / \Sigma_\mu} T_w x T_w^\vee. \]
\end{defn}

\begin{theorem}[Mackey Decomposition, {\cite[Theorem 2.29, 2.30]{Jo}}] \label{thm:mackey}
Let $\mu, \nu \vDash \lambda$. For a $\He_q(\lambda)$-$\He_q(\lambda)$-bimodule $M$,
\[ M \uparrow_{\He_{q}(\mu)}^{\He_{q}(\lambda)}\  \downarrow_{\He_{q}(\nu)} \cong \sum_{w \in \Sigma_\nu \backslash \Sigma_\lambda / \Sigma_\mu} (T_w \otimes_{\He_q(\lambda)} M) \uparrow_{\He_{q}(\nu \cap \mu^w)}^{\He_{q}(\nu)}. \]
Moreover, for an $m \in M^{\He_q(\mu)}$,
\[ \res_{\lambda, \nu} \tr_{\mu, \lambda}(m) = \sum_{w \in \Sigma_\nu \backslash \Sigma_\lambda / \Sigma_\mu} \tr_{\nu \cap \mu^w, \nu} (T_w m T_w^\vee). \]
\end{theorem}

\subsection{} For $\He_q(\lambda)$-modules $M$ and $N$, the transfer map on the $\He_q(\lambda)$-$\He_q(\lambda)$-bimodule $\Hom_{\C}(M, N)$ induces the transfer map on extension groups: 
\[ \tr_{\mu, \lambda} : \Ext^\bullet_{\He_q(\mu)}(M, N) \to \Ext^\bullet_{\He_q(\lambda)}(M, N), \]
and the restriction map on $\Hom_{\C}(M, N)$ induces the restriction map:
\[ \res_{\lambda, \mu} : \Ext^\bullet_{\He_q(\lambda)}(M, N) \to \Ext^\bullet_{\He_q(\mu)}(M, N). \]
The restriction map might be omitted when it is clear from the context. For instance, for $\alpha \in \Ext^\bullet_{\He_q(\lambda)}(M, N)$, we use the shorthand convention: $\tr_{\mu, \lambda}(\alpha) := \tr_{\mu, \lambda} \res_{\lambda, \mu} (\alpha)$.

\begin{prop}\label{prop:transfer}
Let $\lambda$ be a composition, and let $M_1, M_2, M_3$ be three $\He_q(\lambda)$-$\He_q(\lambda)$-bimodules. The following statements hold.
\begin{itemize}
\item[(a)] Let $\mu \vDash \lambda$. For $\alpha \in \Ext^\bullet_{\He_q(\lambda)}(M_1, M_2)$,
\[ \tr_{\mu, \lambda} \res_{\lambda, \mu}(\alpha) = (\tr_{\mu, \lambda} 1_{\He_q(\mu)}) \alpha. \]
\item[(b)] Let $\mu \vDash \lambda$. If $\mu$ is a maximal $l$-parabolic subcomposition, then $\tr_{\mu, \lambda} 1_{\He_q(\mu)}$ is invertible in $\He_q(\lambda)$.
\item[(c)] Let $\mu \vDash \lambda$. For $\alpha \in \Ext^\bullet_{\He_q(\mu)}(M_1, M_2)$, $\beta \in \Ext^\bullet_{\He_q(\lambda)}(M_2, M_3)$,
\[ \beta \circ \tr_{\mu, \lambda}(\alpha) = \tr_{\mu, \lambda}(\res_{\lambda, \mu}(\beta) \circ \alpha). \]
\item[(d)] Let $\nu \vDash \lambda$. For $\alpha \in \Ext^\bullet_{\He_q(\lambda)}(M_1, M_2)$, $\beta \in \Ext^\bullet_{\He_q(\nu)}(M_2, M_3)$,
\[ \tr_{\nu, \lambda}(\beta) \circ \alpha = \tr_{\nu, \lambda}(\beta \circ \res_{\lambda, \nu}(\alpha)). \]
\item[(e)] Let $\nu \vDash \mu \vDash \lambda$. For $\beta \in \Ext^\bullet_{\He_q(\nu)}(M_1, M_2)$,
\[ \tr_{\mu, \lambda} \tr_{\nu, \mu}(\beta) = \tr_{\nu, \lambda}(\beta). \]
\item[(f)] Let $\mu, \nu \vDash \lambda$. For $\alpha \in \Ext^\bullet_{\He_q(\mu)}(M_1, M_2)$,
\[ \res_{\lambda, \nu} \tr_{\mu, \lambda} (\alpha) = \sum_{w \in \Sigma_\nu \backslash \Sigma_\lambda / \Sigma_\mu} \tr_{\nu \cap \mu^w, \nu} (T_w \alpha T_w^\vee). \]
\item[(g)] Let $\mu, \nu \vDash \lambda$. For $\alpha \in \Ext^\bullet_{\He_q(\mu)}(M_1, M_2)$ and $\beta \in \Ext^\bullet_{\He_q(\nu)}(M_2, M_3)$,
\begin{align*}
\tr_{\nu, \lambda}(\beta) \circ \tr_{\mu, \lambda}(\alpha) & = \sum_{w \in \Sigma_\nu \backslash \Sigma_\lambda / \Sigma_\mu} \tr_{\nu \cap \mu^w, \lambda}(\beta \circ (T_w \alpha T_w^\vee)) \\
& = \sum_{w \in \Sigma_\mu \backslash \Sigma_\lambda / \Sigma_\nu} \tr_{\mu \cap \nu^w, \lambda}((T_w \beta T_w^\vee) \circ \alpha).
\end{align*}
\item[(h)] Let $\mu \vDash \lambda$. For $\alpha \in \Ext^\bullet_{\He_q(\mu)}(\C, \C)$ and $\beta \in \Ext^\bullet_{\He_q(\mu)}(\C, M_3)$, there exists $\alpha_\nu \in \Ext^\bullet_{\He_q(\lambda)}(\C, \C)$ for each $\nu \vDash \mu$ such that
\[ \tr_{\mu, \lambda}(\beta \circ \alpha) = \sum_{\nu \vDash \mu} \tr_{\nu, \lambda}(\beta) \circ \alpha_\nu. \]
\end{itemize}
\end{prop}

\begin{proof} (a) The morphism $\alpha \in \Ext^\bullet_{\He_q(\lambda)}(M_1, M_2)$ can be viewed as an element in $\Hom_{\He_q(\lambda)}(M_1', M_2)$ for some $\He_q(\lambda)$-module $M_1'$ via dimension shifting. 
So, for every $w \in \Sigma_\lambda$, $\alpha T_w^\vee = T_w^\vee \alpha$. Therefore,
\[ \tr_{\mu, \lambda} \res_{\lambda, \mu}(\alpha) = \sum_{w \in \Sigma_\lambda / \Sigma_\mu} T_w \alpha T_w^\vee = \sum_{w \in \Sigma_\lambda / \Sigma_\mu} (T_w T_w^\vee) \alpha = (\tr_{\mu, \lambda} 1_{\He_q(\mu)}) \alpha. \]
\vskip .25cm 
\noindent
(b) See \cite[Theorem 2.7]{Du}.
\vskip .25cm 
\noindent
(c) Using the same arguments in (a), we can show that for every $w \in \Sigma_\lambda$, $\beta T_w = T_w \beta$. So,
\[ \beta \circ \tr_{\mu, \lambda}(\alpha) = \beta \circ \sum_{w \in \Sigma_\lambda / \Sigma_\mu} T_w \alpha T_w^\vee = \sum_{w \in \Sigma_\lambda / \Sigma_\mu} T_w (\res_{\lambda, \mu}(\beta) \circ \alpha) T_w^\vee = \tr_{\mu, \lambda}(\res_{\lambda, \mu}(\beta) \circ \alpha). \]
\vskip .25cm 
\noindent
(d) The result follows from similar arguments in (c).
\vskip .25cm 
\noindent
(e) For every $w_1 \in \Sigma_\lambda / \Sigma_\mu$ and $w_2 \in \Sigma_\mu / \Sigma_\nu$, it is clear that $w_1 w_2 \in \Sigma_\lambda / \Sigma_\nu$. So,
\[ \tr_{\mu, \lambda} \tr_{\nu, \mu}(\beta) = \sum_{w_1 \in \Sigma_\lambda / \Sigma_\mu} \sum_{w_2 \in \Sigma_\lambda / \Sigma_\nu} T_{w_1 w_2} \beta T_{w_1 w_2}^\vee = \sum_{w_1 \in \Sigma_\lambda / \Sigma_\nu} T_w \beta T_w^\vee = \tr_{\nu, \lambda}(\beta). \]
\vskip .25cm 
\noindent
(f) Viewing $\alpha$ as an element in $\Hom_{\He(\mu)}(M_1', M_2)$, the result follows from applying Theorem \ref{thm:mackey} to $\Hom_{\He(\mu)}(M_1', M_2)$.
\vskip .25cm 
\noindent
(g) We calculate left hand side of (g) as follow:
\begin{align*}
\tr_{\nu, \lambda}(\beta) \circ \tr_{\mu, \lambda}(\alpha) = & \tr_{\nu, \lambda}(\beta \circ \res_{\lambda, \nu} \tr_{\mu, \lambda}(\alpha)) & \text{using(d)} \\
= & \tr_{\nu, \lambda}\left(\beta \circ \left(\sum_{w \in \Sigma_\nu \backslash \Sigma_\lambda / \Sigma_\mu} \tr_{\nu \cap \mu^w, \nu} (T_w \alpha T_w^\vee)\right)\right) & \text{using (f)} \\
= & \sum_{w \in \Sigma_\nu \backslash \Sigma_\lambda / \Sigma_\mu} \tr_{\nu, \lambda} \tr_{\nu \cap \mu^w, \nu} (\beta \circ (T_w \alpha T_w^\vee)) & \text{using (c)} \\
= & \sum_{w \in \Sigma_\nu \backslash \Sigma_\lambda / \Sigma_\mu} \tr_{\nu \cap \mu^w, \lambda} (\beta \circ (T_w \alpha T_w^\vee)). & \text{using (e)}
\end{align*}
Similarly, using (c), (f), (d) and (e), we get the other identity.
\vskip .25cm 
\noindent
(h) We apply induction on $\mu$ with respect to the partial order $\vDash$. In the case that $\mu$ is the trivial composition $(1, 1, \dots)$, $\Ext^\bullet_{\He_q(\mu)}(\C, \C) \cong \C$, hence the result holds trivially. Suppose that $\mu$ is nontrivial and the result holds for all proper subcomposition of $\mu$. For each $w \in \Sigma_\mu \backslash \Sigma_\lambda / \Sigma_\mu$, since there are no non-trivial normal Young subgroup of $\Sigma_\mu$, $\mu \cap \mu^w$ is a proper subcomposition of $\mu$, hence by induction hypothesis, there exists $\alpha_{w, \nu} \in \Ext^\bullet_{\He_q(\mu \cap \mu^w)}(\C, \C)$ such that
\[ \tr_{\mu \cap \mu^w, \lambda}(\beta \circ (T_w \alpha T_w^\vee))) = \sum_{\nu \vDash \mu \cap \mu^w} \tr_{\nu, \lambda}(\beta) \circ \alpha_{w, \nu}. \]
Then,
\begin{align*}
\tr_{\mu, \lambda}(\beta \circ \alpha) = & \tr_{\mu, \lambda}(\beta) \circ \tr_{\mu, \lambda}(\alpha) - \sum_{w \in \Sigma_\mu \backslash \Sigma_\lambda / \Sigma_\mu \atop w \neq 1} \tr_{\mu \cap \mu^w, \lambda} (\beta \circ (T_w \alpha T_w^\vee)) & \text{using (g)} \\
= & \tr_{\mu, \lambda}(\beta) \circ \tr_{\mu, \lambda}(\alpha) -  \sum_{w \in \Sigma_\mu \backslash \Sigma_\lambda / \Sigma_\mu \atop w \neq 1} \sum_{\nu \vDash \mu \cap \mu^w} \tr_{\nu, \lambda}(\beta) \circ \alpha_{w, \nu} & \text{by induction} \\
= & \tr_{\mu, \lambda}(\beta) \circ \tr_{\mu, \lambda}(\alpha) - \sum_{\nu \vDash \mu} \tr_{\nu, \lambda}(\beta) \circ \left(\sum_{w \in \Sigma_\mu \backslash \Sigma_\lambda / \Sigma_\mu \atop w \neq 1} \alpha_{w, \nu}\right). \qedhere
\end{align*}
\end{proof}

\subsection{}

Let $\mu, \nu \vDash \lambda$ be two subcompositions. For $w \in \Sigma_\nu \backslash \Sigma_\lambda / \Sigma_\mu$ such that $\Sigma_\mu^w = \Sigma_\nu$, let
\[ T_w \cdot \alpha := T_w \alpha T_w^{-1}. \]
This conjugation by $T_w$ gives a map $\Ext^\bullet_{\He_q(\mu)}(M, N) \to \Ext^\bullet_{\He_q(\nu)}(M, N)$ for two $\He_q(\lambda)$-modules $M$ and $N$.

\begin{prop} \label{lem:Desk}
For every $\alpha \in \Ext^\bullet_{\He_q(\lambda)}(M, N)$,
\[ T_w \cdot \res_{\lambda, \mu}(\alpha) = \res_{\lambda, \nu}(\alpha). \]
\end{prop}

\begin{proof}
Since $T_w \in \He_q(\lambda)$, $T_w$ commutes with $\alpha$, so
\[ T_w \cdot \res_{\lambda, \mu}(\alpha) = T_w \res_{\lambda, \mu}(\alpha) T_w^{-1} = T_w T_w^{-1} \res_{\lambda, \nu}(\alpha) = \res_{\lambda, \nu}(\alpha). \qedhere \]
\end{proof}

%%%%%%
%Section 4
%%%%%

\section{Cohomology and Support Varieties}

\subsection{\bf Restriction map}

For each composition $\lambda$, let $R_\lambda := \Ext^\bullet_{\He_q(\lambda)}(\C, \C)$ be the cohomology ring under Yoneda product. For a natural number $d$, set $R_d := R_{(d)}$ for the cohomology ring with respect to the trivial partition $(d)$. 
Given a simple $l$-parabolic subcomposition $\nu$ of $\lambda$, $R_\nu \cong R_l$ and
\[ R_\nu = \begin{cases}
\C[x_\nu] \otimes \Lambda[y_\nu], & l > 2, \\
\C[y_\nu], & l = 2, \\
\end{cases} \]
for some $x_\nu$ and $y_\nu$ such that $\deg x_\nu = 2 l - 2$ and $\deg y_\nu = 2 l - 3$. Set $x_\nu := y_\nu^2$ when $l = 2$. The ring $R_\nu$ has a reduced commutative subring
\[ \widetilde{R}_\nu := \C[x_\nu]. \]
According to Lemma \ref{lem:Desk}, we could choose $x_\nu$ and $y_\nu$ for all simple $l$-parabolic $\nu \vDash \lambda$ compatibly such that $T_w \cdot x_\nu = x_{w \nu}$ and $T_w \cdot y_\nu = y_{w \nu}$ for where $w$ is the double coset representative in $\Sigma_\nu \backslash \Sigma_\lambda / \Sigma_\nu$ and $w \nu \vDash \lambda$ is the unique simple $l$-parabolic subcomposition such that $\Sigma_\nu^w = \Sigma_{w \nu}$.

\begin{theorem} \label{thm:restriction}
Let $\lambda = (\lambda_1, \dots, \lambda_m) \vDash n$ be a composition and set $\lambda / l := (\lfloor \lambda_1 / l \rfloor, \dots, \lfloor \lambda_m / l \rfloor)$. Let $\mu \vDash \lambda$ be a maximal $l$-parabolic subcomposition. The following statements hold.
\begin{itemize}
\item[(a)] The restriction map $\res_{\lambda, \mu}$ induces an isomorphism
\[ \res_{\lambda, \mu}: R_\lambda \xrightarrow{\sim} \left( R_{\nu_1} \otimes \dots \otimes R_{\nu_{|\lambda / l \rfloor|}} \right)^{\Sigma_{|\lambda / l|}} \]
where $\nu_1, \dots, \nu_{\lfloor n / l \rfloor}$ are all simple $l$-parabolic subcomposition of $\mu$. Moreover, the induced isomorphism
\[ \res_{\lambda, \mu}: R_\lambda \xrightarrow{\sim} \left( R_l^{\otimes |\lambda / l |} \right)^{\Sigma_{\lambda / l}} \]
is independent of the choice of $\mu$.
\item[(b)] Under the isomorphism above,
\[ \res_{n, \lambda}: \left( R_l^{\otimes \lfloor n / l \rfloor} \right)^{\Sigma_{\lfloor n / l \rfloor}} \to \left( R_l^{\otimes |\lambda / l |} \right)^{\Sigma_{\lambda / l}} \]
is the restriction of the projection map $R_l^{\otimes \lfloor n / l \rfloor} \to R_l^{\otimes {|\lambda / l|}}$.
\end{itemize}
\end{theorem}

\begin{proof}
(a) Since $\He_q(\lambda) \cong \bigotimes_{i = 1}^m \He_q(\lambda_i)$, by K\"unneth theorem, $R_{\lambda} \cong \bigotimes_{i = 1}^m R_{\lambda_i}$, hence it is enough to prove the result of $\res_{\lambda, \mu}$ for the trivial partition $\lambda = (n)$. Let $\nu \vdash n$ be an $l$-parabolic partition conjugate to $\mu$. The isomorphism induced by $\res_{\lambda, \nu}$ has been proved in \cite{BEM}. By Proposition~\ref{lem:Desk}, $\res_{\lambda, \mu}$ and $\res_{\lambda, \nu}$ induce the same isomorphism.

(b) Let $\nu \vdash n$ be an $l$-parabolic partition conjugate to $\mu$. Then, $R_\lambda \cong R_\mu^{\Sigma_{|\mu / l|}} \cong R_\nu^{\Sigma_{|\nu / l|}}$. So, it suffices to prove the result for partition $\lambda$ where every part is a multiple of $l$, and 
$\mu$ is a maximal $l$-parabolic partition. Since $\res_{(n), \mu} = \res_{\lambda, \nu} \circ \res_{(n), \lambda}$ and the restriction map $\res_{(n), \mu}$ is given by projection, the result follows.
\end{proof}

For a composition $\lambda$, we set
\[ \widetilde{R}_\lambda := \res_{\lambda, \mu}^{-1} \left( \widetilde{R}_{\nu_1} \otimes \dots \widetilde{R}_{\nu_{|\lambda/l|}} \right)^{\Sigma_{\lambda / l}}, \]
where $\nu_1, \dots, \nu_{|\lambda/l|}$ are all simple $l$-parabolic subcomposition of some maximal $l$-parabolic subcomposition $\mu$ of $\lambda$. By Theorem \ref{thm:restriction}, $\widetilde{R}_\lambda$ does not depend on the choice of $\mu$, and it is a commutative reduced subring of $R_\lambda$. Moreover, $R_\lambda$ is a finitely generated $\widetilde{R}_\lambda$-module.

\subsection{\bf Finite Generation:}

For the remainder of this paper, assume that all modules for the Hecke algebra $\He_q(\lambda)$ are finite-dimensional. Let $M \in \mod(\He_q(\lambda))$ and set $\operatorname{H}^\bullet(\He_q(\lambda), M) := \Ext^\bullet_{\He_q(\lambda)}(\C, M)$, which is an $R_\lambda$-module.

\begin{theorem} \label{thm:fg} Let $\lambda$ be a partition of $d$.
\begin{itemize}
\item[(a)] $R_{\lambda}$ is a finitely generated $\C$-algebra.
\item[(b)] If $M\in \operatorname{mod}(\He_q(\lambda))$ then $\operatorname{H}^{\bullet}(\He_q(\lambda), M)$ is a finitely generated $R_{\lambda}$-module.
\end{itemize}
\end{theorem}

\begin{proof}
(a) We can conclude this statement by applying the K\"unneth theorem and \cite[Theorem 1.1]{BEM}.

(b) First consider the case when $\lambda=(l)$. Then one can directly prove using explicit projective resolutions for $\He_{q}(l)$ (cf. \cite[5.1]{KN})
that for any simple $\He_{q}(l)$-module, $S$, then $\operatorname{H}^{\bullet}(\He_q(l), S)$ is a finitely generated $R_{(l)}$-module. Now using 
induction on the composition length and the long exact sequence in cohomology, it follows the statement of (b) holds 
for a finitely generated $\He_{q}(l)$-module $M$. 

Next consider the case when $\lambda=(l^{a},1^{s})$. Any simple $\He_{q}(\lambda)$-module is an outer 
tensor product $S=S_{1}\boxtimes S_{2} \boxtimes \dots \boxtimes S_{a}\boxtimes {\mathbb C}^{s}$ By the K\"unneth theorem,
\[ \operatorname{H}^\bullet(\He_q(\lambda), S) \cong \operatorname{H}^\bullet(\He_q(l)), S_{1})\otimes 
\operatorname{H}^\bullet(\He_q(l), S_{2})\otimes \dots \otimes
\operatorname{H}^\bullet(\He_q(l), S_{a}) . \]
which is a finitely generated $R_{\lambda}$-module from the preceding paragraph. By an inductive argument on the composition length 
the statement holds for $R_{(l^{a}, 1^s)}$. 

Now consider the case when $\lambda=(d)$ and let $\mu=(l^{a},1^{s})$ be a maximal $l$-parabolic partition of $\lambda$. According to 
Proposition~\ref{prop:transfer}(a)(b), the restriction map $\text{res}_{\lambda, \mu} : \HH^\bullet(\He_q(\lambda), M) \to \HH^\bullet(\He_q(\mu), M)$ is injective. 
The codomain $\HH^\bullet(\He_q(\mu), M)$ is already shown to be finitely generated over $R_\mu$. Since $R_{\mu}$ is finitely generated as a 
$R_{\lambda}$-module, it follows that $\HH^\bullet(\He_q(\mu), M)$ is finitely generated over $R_{\lambda}$. Therefore, by Proposition~\ref{prop:transfer}(a) 
the transfer map splits restriction and one can apply Proposition~\ref{prop:transfer}(h) to show that $\HH^\bullet(\He_q(d), M)$ is finitely generated over 
$R_{(d)}$.

Finally, let $\lambda=(\lambda_{1},\lambda_{2},\dots,\lambda_{t})$ be a partition of $d$. Any simple $\He_{q}(\lambda)$-module is an outer 
tensor product $S=S_{1}\boxtimes S_{2} \boxtimes \dots \boxtimes S_{t}$ By the K\"unneth theorem and the fact that the 
finite generation statement holds for $\He(\lambda_{j})$ for $j=1,2,\dots, t$, one can conclude that 
$\operatorname{H}^{\bullet}(\He_q(\lambda), S)$ is a finitely generated $R_{\lambda}$-module. Again by an inductive argument one can now conclude the 
statement of part (b). 
\end{proof}

\subsection{\bf Support Theory:} \label{S:supporttheory} Set $W_{\lambda} := \operatorname{MaxSpec} \widetilde{R}_\lambda$. According to Theorem~\ref{thm:fg} the set $W_{\lambda}$ is an affine homogeneous variety. Given $M \in \text{mod}(\He_q(\lambda))$, define the  (relative) {\it support variety} $W_{\lambda}(M)$ as the variety of the annihilator ideal, $J_{\He_q(\lambda)}(\C, M)$, in $\widetilde{R}_\lambda$ for its action on $\HH^\bullet(\He_q(\lambda), M)$. These support varieties are closed, conical subvarieties of $W_{\lambda}$.

For each $\mu \vDash \lambda$, there exists a restriction map in cohomology $\res_{\lambda, \mu}^\ast : W_{\mu} \to W_{\lambda}$ which is induced by the inclusion of $\He_q(\mu) \leq \He_q(\lambda)$. We can now formulate a definition for the support varieties for modules for $\He_q(\lambda)$.

\begin{defn} \label{lem:vsupport} Let $M \in \mod(\He_q(\lambda))$. 
\begin{itemize} 
\item[(a)] The {\em support variety} of $M$ is defined as
\[ V_{\lambda}(M) := \bigcup_{\mu \vDash \lambda} \res_{\lambda, \mu}^\ast (W_{\mu}(M)). \]
\item[(b)] In the case when $\lambda=(d)$, 
\[ V_{\He_{q}(d)}(M) :=V_{(d)}(M)=\bigcup_{\mu \vDash (d)} \res_{(d), \mu}^\ast (W_{\mu}(M)). \]
\end{itemize} 
\end{defn}

By using the functoriality of the restriction map and the fact that the restriction maps are finite maps, one can state the following proposition. 

\begin{prop} \label{prop:Tavern}
Let $W$ be closed subvariety of $W_\nu$.
\begin{itemize}
\item[(a)] $\dim W = \dim \res_{\lambda, \nu}^\ast(W)$.
\item[(b)] $\res_{\lambda, \nu}^\ast = \res_{\lambda, \mu}^\ast \circ \res_{\mu, \nu}^\ast$.
\end{itemize}
\end{prop}

Next we present below several elementary properties of these support varieties. The proofs from \cite[\S 5.7]{Ben} can be used to verify these facts.

\begin{prop} \label{prop:Bracelet} Let $M_{j}\in \operatorname{mod}(\He_{q}(d))$ for $j=1,2,3$. Then 
\begin{itemize}
\item[(a)] Let $0\rightarrow M_{1}\rightarrow M_{2} \rightarrow M_{3}
\rightarrow 0$ be a short exact sequence  in $\operatorname{mod}(\He_q(\lambda))$. If $\Sigma_3$ is the symmetric group on three letters and
$\sigma\in \Sigma_3$ then
$$V_{\lambda}(M_{\sigma(1)}) \subseteq V_{\lambda}(M_{\sigma(2)}) \cup V_{\lambda}(M_{\sigma(3)}).$$
\item[(b)]  $V_{\lambda}(M_{1}\oplus M_{2}) = V_{\lambda}(M_1) \cup V_{\lambda}(M_2)$.
\end{itemize}
\end{prop}

The following proposition gives an simplification of the formulas given in Definition~\ref{lem:vsupport} via maximal $l$-parabolic subcompositions. 

\begin{prop} Let $\mu \vDash \lambda$ be a maximal $l$-parabolic subcomposition, and let $M \in \mod(\He_q(\lambda))$.
\begin{itemize}
\item[(a)] For every maximal $l$-parabolic subcomposition $\mu \vDash \lambda$, $W_\lambda(M) = \res_{\lambda, \mu}^\ast W_\mu(M)$.
\item[(b)]
\[ V_{\lambda}(M) = \bigcup_{\text{$l$-parabolic} \atop \mu \vDash \lambda} \res_{\lambda, \mu}^\ast (W_{\mu}(M)). \]
\item[(c)] For any maximal $l$-parabolic subcomposition $\mu \vDash \lambda$,
\[ V_\lambda(M) = \res_{\lambda, \mu}^\ast V_\mu(M). \]
\end{itemize}
\end{prop}

\begin{proof}
\noindent (a) Consider the transfer map
\[ \tr_{\mu, \lambda} : J_{\He_q(\mu)}(\C, M) \to J_{\He_q(\lambda)}(\C, M) \]
and the restriction map
\[ \res_{\lambda, \mu} : J_{\He_q(\lambda)}(\C, M) \to J_{\He_q(\mu)}(\C, M) \]
According to \ref{prop:transfer}(a) and (b), $\tr_{\mu, \lambda} \circ \res_{\lambda, \mu} = a\ {\rm id}$ for some unit $a \in \He_q(\lambda)$ as an endomorphism of $J_{\He_q(\lambda)}(\C, M)$. One has $\res^\ast_{\lambda, \mu} \circ \tr^\ast_{\mu, \lambda} = {\rm id}$ as an endomorphism of $W_\mu(M)$, from which the result follows.

\noindent (b) For each $\mu \vDash \lambda$, let $\mu'$ be a maximal $l$-parabolic subcomposition of $\mu$. Therefore, by (a),
\begin{align*}
V_{\lambda}(M) = & \bigcup_{\mu \vDash \lambda} \res_{\lambda, \mu}^\ast (W_{\mu}(M)) \\
= & \bigcup_{\mu \vDash \lambda} \res_{\lambda, \mu}^\ast \res_{\mu, \mu'}^\ast (W_{\mu'}(M)) \\
= & \bigcup_{\text{$l$-parabolic} \atop \mu' \vDash \lambda} \res_{\lambda, \mu'}^\ast (W_{\mu'}(M)).
\end{align*}

\noindent (c) The result follows from (b) and the fact that every $l$-parabolic subcomposition of $\lambda$ is contained in a given maximal $l$-parabolic subcomposition up to conjugacy.
\end{proof}

\subsection{\bf Varieties and Induction:}

The following proposition states how relative support behave under induction. 

\begin{prop} \label{prop:induction} Let $\nu, \mu, \lambda$ be three compositions such that $\mu \vDash \lambda$ and $M \in \mod(\He_q(\mu))$.
\begin{itemize}
\item[(a)] $W_\lambda(M \uparrow_{\He_{q}(\mu)}^{\He_{q}(\lambda)}) = \res_{\lambda, \mu}^\ast W_\mu(M)$.
\item[(b)] $W_\nu(M \uparrow_{\He_{q}(\mu)}^{\He_{q}(\lambda)}) = \bigcup_{\alpha \vDash \mu} \bigcup_{w \in w_{\alpha, \mu, \nu}} T_w^\# \res_{\nu^w, \alpha}^\ast W_\alpha(M)$, where
\[ w_{\alpha, \mu, \nu} := \{w \in \Sigma_{\nu} \backslash \Sigma_\lambda / \Sigma_{\mu} : \alpha = \mu \cap \nu^w\}. \]
\item[(c)] $V_\lambda(M \uparrow_{\He_{q}(\mu)}^{\He_{q}(\lambda)}) = \res_{\lambda, \mu}^\ast V_\mu(M)$.
\end{itemize}
\end{prop}

\begin{proof}
(a) This follows by applying Frobenius reciprocity: $\Ext^\bullet_{\He_q(\lambda)}(\C, M \uparrow_{\He_{q}(\mu)}^{\He_{q}(\lambda)}) \cong \Ext^\bullet_{\He_q(\mu)}(\C, M)$.

(b) The result follows from the following calculation.
\begin{align*}
W_\nu(M \uparrow_{\He_{q}(\mu)}^{\He_{q}(\lambda)}) &= W_\nu\left( \bigoplus_{w \in \Sigma_{\nu} \backslash \Sigma_\lambda / \Sigma_{\mu}} (T_w \otimes M) \uparrow_{\He_{q}(\mu^w \cap \nu)}^{\He_{q}(\nu)} \right) & \text{Theorem \ref{thm:mackey}} \\
& = \bigcup_{w \in \Sigma_{\nu} \backslash \Sigma_\lambda / \Sigma_{\mu}} W_\nu((T_w \otimes M) \uparrow_{\He_{q}(\mu^w \cap \nu)}^{\He_{q}(\nu)}) & \text{Proposition \ref{prop:Bracelet}(b)} \\
& = \bigcup_{w \in \Sigma_{\nu} \backslash \Sigma_\lambda / \Sigma_{\mu}} \res_{\nu, \mu^w \cap \nu}^\ast W_{\mu^w \cap \nu}(T_w \otimes M) \\
& = \bigcup_{w \in \Sigma_{\mu} \backslash \Sigma_\lambda / \Sigma_{\nu}} T_w^\# \res_{\nu^w, \mu \cap \nu^w}^\ast W_{\mu \cap \nu^w}(M) \\
& = \bigcup_{\alpha \vDash \mu} \bigcup_{w \in w_{\alpha, \mu, \nu}} T_w^\# \res_{\nu^w, \alpha}^\ast W_\alpha(M).
\end{align*}

(c) We proceed with the following calculation.
\begin{align*}
V_\lambda(M \uparrow_{\He_{q}(\mu)}^{\He_{q}(\lambda)}) & = \bigcup_{\nu \vDash \lambda} \res_{\lambda, \nu}^\ast W_\nu(M \uparrow_{\He_{q}(\mu)}^{\He_{q}(\lambda)})& \text{Definition \ref{lem:vsupport}} \\
& = \bigcup_{\nu \vDash \lambda} \res_{\lambda, \nu}^\ast \bigcup_{\alpha \vDash \mu} \bigcup_{w \in w_{\alpha, \mu, \nu}} T_w^\# \res_{\nu^w, \alpha}^\ast W_\alpha(M) & \text{by (a)} \\ 
& = \bigcup_{\nu \vDash \lambda} \bigcup_{\alpha \vDash \mu} \bigcup_{w \in w_{\alpha, \mu, \nu}} \res_{\lambda, \nu}^\ast T_w^\# \res_{\nu^w, \alpha}^\ast W_\alpha(M) \\
& = \bigcup_{\nu \vDash \lambda} \bigcup_{\alpha \vDash \mu} \bigcup_{w \in w_{\alpha, \mu, \nu}} T_w^\# \res_{\lambda^w, \nu^w}^\ast \res_{\nu^w, \alpha}^\ast W_{\alpha}(M) \\
& = \bigcup_{\nu \vDash \lambda} \bigcup_{\alpha \vDash \mu} \bigcup_{w \in w_{\alpha, \mu, \nu}} T_w^\# \res_{\lambda, \alpha}^\ast W_\alpha(M) & w \in \Sigma_\lambda \\
& = \bigcup_{\nu \vDash \lambda} \bigcup_{\alpha \vDash \mu} \bigcup_{w \in w_{\alpha, \mu, \nu}} \res_{\lambda, \alpha}^\ast W_\alpha(M) & T_w^\# \in \He_q(\lambda) \\
& = \bigcup_{\alpha \vDash \mu} \bigcup_{\nu \vDash \lambda \atop w \in w_{\alpha, \mu, \nu}} \res_{\lambda, \alpha}^\ast W_{\alpha}(M) \\
& = \bigcup_{\alpha \vDash \mu} \res_{\lambda, \alpha}^\ast W_{\alpha}(M) & \text{for $\nu = \mu$, $1 \in w_{\alpha, \mu, \nu}$}\\
& = \res_{\lambda, \mu}^\ast V_\mu(M). \qedhere
\end{align*}
\end{proof}

We end this section with a result that will useful for computing support varieties in the case when one has some information about the vertex of a module. In particular this will be applied 
in case of Young vertices. 

\begin{prop} \label{prop:go}
Let $\mu \vDash \lambda$. Suppose that $M$ is an $\He_q(\lambda)$-module and $N$ is an $\He_q(\mu)$-module such that $M \mid N \uparrow_{\He_{q}(\mu)}^{\He_{q}(\lambda)}$ and $N \mid M \downarrow_{\He_{q}(\mu)}$. Then, $V_\lambda(M) = \res_{\lambda, \mu}^\ast V_\mu(N)$.
\end{prop}

\begin{proof}
Using Proposition \ref{prop:Bracelet}(b) and \ref{prop:induction}(c), we obtain
\[ V_\lambda(M) \subseteq V_\lambda(N \uparrow_{\He_{q}(\mu)}^{\He_{q}(\lambda)}) = \res_{\lambda, \mu}^\ast V_\mu(N). \]
It follows from the Definition \ref{lem:vsupport} that
\[ \res_{\lambda, \mu}^\ast V_\mu(N) \subseteq \res_{\lambda, \mu}^\ast V_\mu(M) \subseteq V_\lambda(M). \qedhere \]
\end{proof}

%%%%%%
%%Section 5
%%%%%%

\section{Rates of Growth}

\subsection{\bf Complexity:} Let $\{d_n\}_{n\geq 0}$ be a sequence of non-negative integers.
The {\em rate of growth} $r(d_{\bullet})$ of this sequence is the smallest non-negative integer $c$  for which there
exists a positive real number $C$ such that
$d_{n}\leq C\cdot n^{c-1}$ for all $ n\geq 1$.
If no such $d$ exists, set $r(d_{\bullet}):=\infty$.

Alperin  \cite[\S 4]{A} first defined the notion of complexity of modules for finite groups. We can also state this for Hecke algebras.
Our goal will be to relate the complexity to the dimension of the support varieties defined in the previous section.

\begin{defn} Let $M \in {\rm mod}(\He_q(d))$ and let
$$\dots \rightarrow P_{2} \rightarrow P_{1} \rightarrow P_{0}
\rightarrow M \rightarrow
0$$
be the minimal projective resolution of $M$. The {\em complexity} $c_{\He_q(d)}(M)$ of $M$ is
defined as $r(\dim P_{\bullet})$.
\end{defn}

\subsection{} For Hecke algebras the conventional proofs to relate the dimension of the support variety to (i) the
rate of growth of certain extension groups and (ii) the complexity of the module do not work because of the absence of
the tensor product (i.e., a coproduct structure on  $\He_q(d)$).

We first prove that the complexity can still be interpreted as the rate of growth of certain $\Ext$-groups related to taking
the direct sum of simple, Specht, Young and permutation modules.

\begin{theorem} \label{thm:comparerates} Let $M\in \operatorname{mod}(\He_q(d))$. The following quantities are equal.
\label{equalrateofgrowth}
\begin{itemize}
\item[(a)] $c_{\Sigma_{d}}(M)$;
\item[(b)] $r(\Ext^{\bullet}_{\He_q(d)}(\oplus_{\lambda\in \Lambda^+_{\text{reg}}(d)} D^{\lambda},M))$;
\item[(c)] $r(\Ext^{\bullet}_{\He_q(d)}(\oplus_{\lambda\vdash d} S^{\lambda},M))$;
\item[(d)] $r(\Ext^{\bullet}_{\He_q(d)}(\oplus_{\lambda\vdash d} Y^{\lambda},M))$;
\item[(e)] $r(\Ext^{\bullet}_{\He_q(d)}(\oplus_{\lambda\vdash d} M^{\lambda},M))$.
\end{itemize}
\end{theorem}

\begin{proof} $\text{(a)}=\text{(b)}$. This follows by using the standard arguments (cf. \cite[Prop. 5.3.5]{Ben}).

$\text{(c)}\leq \text{(b)}$, $\text{(d)}\leq \text{(b)}$.  One can apply \cite[Prop. 5.3.5]{Ben} to deduce these statements.

$\text{(b)}\leq \text{(c)}$. This will be proved by using induction on
the dominance order of partitions. Set $s:=r(\Ext^{\bullet}_{\He_q(d)}(\oplus_{\lambda\vdash d}S^{\lambda},M))$.
If $\lambda$ be maximal with respect to $\unlhd$ then $S^{\lambda}=D^{\lambda}$. Consequently, 
$$r(\Ext^{\bullet}_{\He_q(d)}(D^{\lambda},M))\leq s.$$ By induction suppose that for every $\mu \rhd \tau$, we know
$r(\Ext^{\bullet}_{\He_q(d)}(D^{\mu},M))\leq s$. We need to show that $r(\Ext^{\bullet}_{\He_q(d)}(D^{\tau},M))\leq s$.
There exists a short exact sequence of the form
\begin{equation}
0\rightarrow N \rightarrow S^{\tau}\rightarrow D^{\tau}\rightarrow 0
\end{equation}
with $N$ having composition factors of the form $D^{\mu}$ with
$\mu \rhd \tau$. Therefore, 
$$r(\Ext^{\bullet}_{\He_q(d)}(D^{\tau},M))\leq
\text{max}\{ r(\Ext^{\bullet}_{\He_q(d)}(S^{\tau},M)), r(\Ext^{\bullet}_{\He_q(d)}(N,M)) \} \leq s.$$

$\text{(c)}\leq \text{(d)}$. This statement will be proved in a similar fashion
as above. Set
$$y:=r(\Ext^{\bullet}_{\He_q(d)}(\oplus_{\lambda\vdash d}Y^{\lambda},M)).$$
Let $\lambda$ be maximal with respect to $\unlhd$ so $Y^{\lambda}=S^{\lambda}$ and
$r(\Ext^{\bullet}_{\He_q(d)}(S^{\lambda},M))\leq y$.
Suppose that for any $\mu \rhd \tau$, $r(\Ext^{\bullet}_{\He_q(d)}(S^{\mu},M))\leq y$.
It will suffice to show that $r(\Ext^{\bullet}_{\He_q(d)}(S^{\tau},M))\leq y$.
There is a short exact sequence of the form
\begin{equation}\label{eq:Young-Specht}
0\rightarrow S^{\tau} \rightarrow Y^{\tau}\rightarrow Z \rightarrow 0
\end{equation}
with $Z$ having a Specht filtration with factors of the form $S^{\mu}$ with
$\mu \rhd \tau$. Consequently,
$$r(\Ext^{\bullet}_{\He_q(d)}(S^{\tau},M))\leq
\text{max}\{ r(\Ext^{\bullet}_{\He_q(d)}(Y^{\tau},M)), r(\Ext^{\bullet}_{\He_q(d)}(N,M)) \} \leq y.$$

$\text{(d)} = \text{(e)}$. The statement follows because every Young module appears as a summand of a permutation module, and the summands of the
permutation modules are Young modules.
\end{proof}

\subsection{\bf Complexity and Support Varieties} We now can relate the complexity of modules in $\text{mod}(\He_q(d)))$ to the dimension of their
support varieties. Furthermore, every module in $\He_q(d))$ has complexity less than the complexity of the trivial module. Note that
for Hopf algebras this is an easy consequence of tensoring a minimal projective resolution of the trivial module by the given module $M$.

\begin{cor} \label{c:com-support} Let $M\in \operatorname{mod}(\He_q(d))$. Then
\begin{itemize}
\item[(a)] $c_{\He_q(d)}(M)=\dim V_{\He_q(d)}(M)$
\item[(b)] $c_{\He_q(d)}(M)\leq c_{\He_q(d)}(\C)$
\end{itemize}
\end{cor}

\begin{proof} (a) Since $\res_{d, \lambda}^\ast$ is a finite map, $\dim \res_{(d), \lambda}^\ast W_\lambda(M) = \dim W_\lambda(M)$. 

Next by using the argument given in \cite[p. 105-106]{Evens} one has 
$$r(\Ext^{\bullet}_{\He_q(\lambda)}(\C, M))= \dim W_\lambda(M).$$ 
According to Theorem~\ref{thm:comparerates} and Frobenius reciprocity,
\begin{eqnarray*}
c_{\He_q(d)}(M) & = & r(\Ext^{\bullet}_{\He_q(d)}(\oplus_{\lambda\vdash d} M^{\lambda}, M)) \\
& = & r(\oplus_{\lambda \vdash d} \Ext^{\bullet}_{\He_q(\lambda)}(\C, M)) \\
& = & \max_{\lambda \vdash d} \{r(\Ext^{\bullet}_{\He_q(\lambda)}(\C, M))\} \\
& = & \max_{\lambda \vdash d} \{\dim W_\lambda(M)\} \\
& = & \max_{\lambda \vdash d} \{\dim \res_{(d), \lambda}^\ast W_\lambda(M)\} \\
& = & \dim \bigcup_{\lambda \vdash d} \res_{(d), \lambda}^\ast W_\lambda(M) \\
& = & \dim V_{\He_q(d)}(M).
\end{eqnarray*}

(b) From part (a), $c_{\He_q(d)}(M) = \dim V_{\He_q(d)}(M) \leq \dim V_{\He_q(d)}(\C) = c_{\He_q(d)}(\C)$.
\end{proof}

%%%%%%%
%%Section 6
%%%%%%%
\section{Permutation and Young modules}

\subsection{}  In this section we will use our established properties on complexity and support varieties, in addition to the theory of Young vertices, to give an formula for the complexities of the permutation 
modules $\{M^\lambda\}$ and the Young modules $\{Y^\lambda\}$. This is accomplished by first determining their support varieties as images of the map $\res_{d, \lambda}^\ast$ (resp. 
$\res_{d, \rho(\lambda)}^\ast$) on the support varieties of the trivial module.  

Let $\lfloor\  \rfloor$ denote the floor function. Note that the Krull dimension of the cohomology ring $\HH^{\bullet}(\He_q(d)),\C)$ or equivalently 
$\dim V_{\He_q(d)}(\C)$ is $\lfloor d/l \rfloor$. We can now determine the complexity and support varieties for the permutation modules $M^\lambda$:

\begin{prop} Let $\lambda=(\lambda_1,\ldots , \lambda_s) \vDash d$ and $M^{\lambda}$ be a permutation
module for $\He_q(d)$. Then:
\begin{itemize}
\item[(a)] $V_{\He_q(d)}(M^{\lambda}) = \res_{(d), \lambda}^\ast (V_{\He_q(\lambda)}(\C))$;
\item[(b)] $c_{\He_q(d)}(M^{\lambda})  = \sum_{i = 1}^s \lfloor \lambda_i / l \rfloor$.
\end{itemize}
\end{prop}

\begin{proof} Part (a) follows immediately from Proposition \ref{prop:go} since $M^\lambda \cong \C \uparrow^{\He_q(d)}_{\He_q(\lambda)}$ and $\C$ is a direct summand of $M^\lambda \downarrow_{\He_{q}(\lambda)}$ by Theorem \ref{thm:mackey}. 
One can deduce (b) follows from (a) since the map $\res_{(d), \lambda}^\ast$ preserves
dimension and $\dim(V_{\He_q(\lambda)}(\C))$ is determined by Theorem~\ref{thm:restriction}.
\end{proof}

\subsection{} Dipper-Du \cite[5.8 Theorem]{DD} determines the vertices of the Young module $Y^{\lambda}$ for $\He_{q}(d)$ as $\He_{q}(\rho(\lambda))$ where $\rho(\lambda)$ is constructed as follows. 
Notice that any $\lambda \vdash d$ has a unique $l$-adic expansion of the form:
\begin{equation}
\lambda = \lambda_{[0]} + \lambda_{[1]} l,
\end{equation}
where $\lambda_{[0]}$ is an $l$-restricted partition of $d$ and $\lambda_{[1]}$ is a partition of $d$. Define the partition:
\begin{equation} \label{rho}
\rho(\lambda) := \left( l^{|\lambda_{[1]}|}, 1^{|\lambda_{[0]}|} \right).
\end{equation}
The partition $\lambda_{[0]}$ can be obtained by successively striping horizontal rim $l$-hooks from $\lambda$, and $|\lambda_{[1]}|$ is the number of such hooks removed. The following theorem demonstrates that the complexity of the Young module $Y^{\lambda}$ is $|\lambda_{[1]}|$.

\begin{theorem} \label{youngcomplexity}
Let $\lambda \vdash d$ with $Y^{\lambda}$ the corresponding Young module for $\He_q(\lambda)$. Then
\begin{itemize}
\item[(a)] $V_{\He_q(d)}(Y^{\lambda}) = \res_{d, \rho(\lambda)}^\ast V_{\He_q(\rho(\lambda))}(\C)$.
\item[(b)] $c_{\He_q(d)}(Y^{\lambda}) = |\lambda_{[1]}|$.
\end{itemize}
\end{theorem}
\begin{proof} Part (a) follows from Proposition \ref{prop:go}. In order to prove (b) take the dimension on both sides of (a) and recall from Proposition \ref{prop:Tavern} that $\res_{(d), \rho(\lambda)}^\ast$ preserves dimension, and the dimension of the support variety of the trivial module is also determined in Theorem~\ref{thm:restriction}.
\end{proof}

As a consequence of the aforementioned theorem, we recover the well-known fact that $Y^\lambda$ is projective exactly when $\lambda$ is $p$-restricted. Furthermore, from Theorem \ref{youngcomplexity}(b), 
one can see that for a block $\mathbb{B}$ of weight $w$, there are Young modules in $\mathbb{B}$ of every possible complexity $\{0,1, \ldots , w\}$. The following result characterizes 
Young module of complexity one. 

\begin{cor}A nonprojective Young module $Y^\lambda$ has complexity one if and only if $\lambda$ is of the form
$(\mu_1+l,\mu_2, \ldots , \mu_s)$ where $(\mu_1,\mu_2, \ldots ,\mu_s)$ is $l$-restricted.
\end{cor}

\begin{proof} From Theorem \ref{youngcomplexity}, $\lambda_{[1]}=1$ precisely when the $l$-adic expansion of $\lambda$ has the form $\lambda_{[0]} + (1)l$. 
\end{proof}

Recall that a module is called {\em periodic} if it admits a periodic projective resolution. From the definition of complexity it is easy to see 
that if $M$ is a finite-dimensional periodic $\He_{q}(d)$-module then $c_{\He_{q}(d)}(M)=1$. However, it is not clear whether the converse holds. 
For finite groups the known proofs depend on using the coalgebra structure on the group algebra which allows one to put an action on the tensor products of modules.

\subsection{} In this section we will apply our prior computation for Young modules to give an explicit 
description for the location of the support varieties for modules in 
a block ${\mathbb B}$ of $\He_{q}(d)$. For a Specht module $\He_{q}(d)$-module, $S_{\lambda}$, let ${\mathbb B}_{\lambda}$ be the 
block of $\He_{q}(d)$ containing $S_{\lambda}$. We remark that all the composition factors of a given Specht module lie in the same block. 
Note that by Nakayama rule, ${\mathbb B}_{\lambda}={\mathbb B}_{\mu}$ if and only if they have the same $l$-core. 

Let $d=c_{[0]}+c_{[1]}l$ be the unique $p$-adic expansion of $d$, so $0 \leq c_{[0]}<l$, and $d=a_{[0]}+a_{[1]}l$ is another expansion, with $0 \leq a_{[0]}$.  Then 
$a_{[0]}\geq c_{[0]}$ and $a_{[1]}\leq c_{[1]}$ and 
$$\He_{q}((l^{a_{[1]}},1^{a_{[0]}})) \leq \He_{q}((l^{c_{[1]}},1^{c_{[0]}})).$$ 

Now suppose $\mathbb{B}_{\mu}$ is a block of $\He_{q}(d)$ with weight $w$ and $l$-core $\tilde{\mu}\vdash d-lw$.  Let $lw=c_{[1]}l$ and 
\begin{equation}
\label{defrho}
\rho_{\text{max}}=(1^{d-lw},l^{w}) \vdash d.
\end{equation}

Let $\tilde{\mu}=(\tilde{\mu}_1, \tilde{\mu}_2, \ldots)$. The algebra $\He_{q}(\rho_{\text{max}})$ is the Young vertex for $Y^\mu$ where $\mu=(\tilde{\mu}_1+lw, \tilde{\mu}_2, \ldots)$. 
Furthermore, if ${\mathbb B}_{\lambda} ={\mathbb B}_\mu$, then $\mu \unrhd \lambda$ and the Young vertex of $Y^\lambda$ is of the form
$$\rho(\lambda)= (1^{a_{[0]}},l^{a_{[1]}})$$ where $a_{[0]} \geq d-lw$ and $a_{[1]}\leq w$. Therefore, 
$$\He_{q}({\rho(\lambda)}) \leq \He_{q}(\rho_{\text{max}})$$
That is, the Young vertices for the Young modules in a block are all contained in a unique vertex $\He_{q}(\rho_{\max})$, which is the vertex for the Young module $Y^{\tilde{\mu}+(lw)}$. 

Define the support of the block to be  $V_{\He_{q}(d)}({\mathbb B}_{\lambda}):=V_{\He_{q}(d)}(\oplus_{\mu\in {\mathbb B}_{\lambda}}D^{\mu})$
We now give a precise location for the support variety for a block of the Hecke algebra $\He_{q}(d)$. 

\begin{theorem} \label{thm:location} Let ${\mathbb B}_{\lambda}$ be a block of $\He_{q}(d)$ of weight $w$ and 
let $M$ be a finite-dimensional module in ${\mathbb B}_{\lambda}$. Let $\rho:=\rho_{\operatorname{max}}$ for the block 
${\mathbb B}_{\lambda}$. 
Then:
\begin{itemize} 
\item[(a)] $V_{\He_{q}(d)}({\mathbb B}_{\lambda})=V_{\He_{q}(d)}(\oplus_{\mu\in {\mathbb B}_{\lambda}} S^{\mu})=
V_{\He_{q}(d)}(\oplus_{\mu\in {\mathbb  B}_{\lambda}} Y^{\mu})$;
\item[(b)] $\operatorname{res}_{(d),\rho}(V_{\He_{q}(\rho)}(k))=V_{\He_{q}(d)}({\mathbb B}_{\lambda})$;  
\item[(c)] $V_{\He_{q}(d)}(M)\subseteq \operatorname{res}_{(d),\rho}(V_{\He_{q}(\rho)}(k))$;
\item[(d)] $c_{\He_{q}(d)}(M)\leq w$. 
\end{itemize} 
\end{theorem} 

\begin{proof} (a): Since $S^{\mu}$ has a filtration with sections being irreducible module and 
$Y^{\mu}$ has a filtration with sections being Specht modules, one has using the definition of support in 
Section \ref{S:supporttheory}, 
$$V_{\He_{q}(d)}({\mathbb B}_{\lambda})\subseteq V_{\He_{q}(d)}(\oplus_{\mu\in {\mathbb B}_{\lambda}} S^{\mu})\subseteq 
V_{\He_{q}(d)}(\oplus_{\mu\in {\mathbb B}_{\lambda}} Y^{\mu}).$$ 

For the other inclusion, one needs to apply the ordering of factors on these filtrations. 
From Theorem \ref{thm:comparerates}, we have exact sequences of the form 
\begin{equation} 
0\rightarrow N \rightarrow S^{\tau}\rightarrow D^{\tau} \rightarrow 0 
\end{equation} 
where the composition factors in $N$ are of the form $D^{\mu}$ with $\mu \rhd \tau$. By induction 
we can assume that $V_{\He_{q}(d)}(N)\subseteq V_{\He_{q}(d)}(\oplus_{\mu\in {\mathbb B}_{\lambda}} S^{\mu})$
and Proposition \ref{prop:Bracelet}, it follows that 
$$V_{\He_{q}(d)}(D^{\tau})\subseteq V_{\He_{q}(d)}(S^{\tau})\cup V_{\He_{q}(d)}(N)\subseteq V_{\He_{q}(d)}(\oplus_{\mu\in {\mathbb B}_{\lambda}} S^{\mu}).$$
Therefore, $V_{\He_{q}(d)}(\oplus_{\mu\in {\mathbb B}_{\lambda}} S^{\mu})\subseteq V_{\He_{q}(d)}({\mathbb B}_{\lambda})$. 
A similar inductive argument using (\ref{eq:Young-Specht}) can be used to prove that 
$$V_{\He_{q}(d)}(\oplus_{\mu\in {\mathbb B}_{\lambda}} Y^{\mu})\subseteq V_{\He_{q}(d)}({\mathbb B}_{\lambda})\subseteq V_{\He_{q}(d)}(\oplus_{\mu\in {\mathbb B}_{\lambda}} S^{\mu}).$$ 

(b): From (a), $V_{\He_{q}(d)}({\mathbb B}_{\lambda})=V_{\He_{q}(d)}(\oplus_{\lambda\in {\mathbb B}_{\mu}}Y^{\lambda})$. Now by analysis prior to the 
statement of the theorem,
$$V_{\He_{q}(d)}(\oplus_{\lambda\in {\mathbb B}_{\mu}}Y^{\lambda})=V_{\He_{q}(d)}(Y^{\rho})=\text{res}_{(d),\rho}(V_{\He_{q}(\rho)}(k)).$$

(c) and (d):  This follows because for any $M$ in ${\mathbb B}_{\lambda}$, 
$V_{\He_{q}(d)}(M)\subseteq V_{\He_{q}(d)}({\mathbb B}_{\lambda})$ by Proposition \ref{prop:Bracelet}. Part (d) follows by considering dimension and applying parts (b) and (c). 
\end{proof} 

%%%%%%%
%%Section 7
%%%%%%%

\section{Specht Modules, Vertices, and Cohomology}

\subsection{} In this section, we will consider the question of computing vertices for Specht modules. This will entail introducing a graded dimension for Specht modules, in addition to, 
considering the relative cohomology for Hecke algebras of Young subgroups. 

\subsection{Weights of Partitions}

For a partition $\lambda$ and a natural number $l$, the {\em $l$-weight} of $\lambda$, denoted by ${\rm wt}_l \lambda$, is the number of $l$-hooks that we could consecutively remove from the partition $\lambda$ to reach 
the {\em $l$-core} of $\lambda$, denoted by ${\rm core}_l \lambda$. For a natural number $n$, we set the $l$-weight of $n$ to be the $l$-weight of the trivial partition $(n)$, so ${\rm wt}_l n = {\rm wt}_l (n) = \lfloor n / l \rfloor$. 
It is clear that $|\lambda| = |{\rm core}_l \lambda| + l {\rm wt}_l \lambda$. We say that $\lambda$ has {\em small $l$-core} if $|{\rm core}_l \lambda| < l$.

\begin{lemma} \label{lem:Conductor}
Let $\lambda$ be a partition and $l$ be a natural number. The number of hooks whose lengths are multiple of $l$ is ${\rm wt}_l(\lambda)$, the $l$-weight of $\lambda$.
\end{lemma}

\begin{proof}
We will prove the result with the help of $l$-abacus of partition $\lambda$. Suppose that $\lambda = (\lambda_1, \dots, \lambda_r)$, and let $b_i := \lambda_i - b + r$. The beads of the $l$-abacus occupies positions $b_i$. Hooks of length multiple of $l$ are in bijection with the moves of a bead at $b_i$ to an unoccupied position $b_i - l k$, which is in the same runner with $b_i$, for some $k \geq 1$. The number of such moves is exactly the $l$-weight of $\lambda$.
\end{proof}

\subsection{Dimensions of Specht Modules}

For an integer $n$, let the $t$-integer be $[n]_t := \frac{1 - t^n}{1 - t}$. When $t = 1$, one applies limits to obtain $[n]_{1} = n$. We will now define 
a graded version of the dimension for Specht modules (also referred to as the {\em graded dimension}) that will involve the divisibility of cyclotomic polynomials. For a partition $\lambda$, let
\[ \dim_t S^\lambda := \frac{\prod_{i = 1}^{|\lambda|} [i]_t}{\prod_{i \in I} [h_i]_t}, \]
where $I$ is the set of all hooks of $\lambda$ and $h_i$ is the hook length of the hook $i$. By hook length formula, we have $\dim_1 S^\lambda = \dim S^\lambda$. The graded dimension of the partition $\lambda$ is the generic degree of the partition $\lambda$ up to a power of $t$ \cite[\textsection 13.5]{Car}, and the graded dimension is a polynomial with nonnegative integer coefficients \cite[\textsection III.6]{Mac}.

\begin{theorem} \label{thm:Sailor}
Let $\Phi_l(t)$ be the $l$-th cyclotomic polynomial in $t$. Then,
\[ \dim_t S^\lambda = \prod_l \Phi_l(t)^{{\rm wt}_l |\lambda| - {\rm wt}_l \lambda} = \prod_l \Phi_l(t)^{{\rm wt}_l |{\rm core}_l \lambda|}, \]
where $l$ runs over all natural number. In particular,
\[ \dim S^\lambda = \prod_{p, r} p^{{\rm wt}_{p^r} |\lambda| - {\rm wt}_{p^r} \lambda} = \prod_{p, r} p^{{\rm wt}_{p^r} |{\rm core}_{p^r} \lambda|}, \]
where $p$ runs over all primes and $r$ runs over all natural number.
\end{theorem}

\begin{proof}
Let $l$ be an arbitrary natural number. When $l = 1$, there are no factors $\Phi_1(t)$ in $\dim_t S^\lambda$, and ${\rm wt}_1 |\lambda| - {\rm wt}_1 \lambda = 0$. Now, assume that $l \geq 2$. Applying Lemma \ref{lem:Conductor} to the trivial partition $(|\lambda|)$, the number of times $\Phi_l(t)$ dividing the numerator of $\dim_t S^\lambda$ is ${\rm wt}_l (|\lambda|) = {\rm wt}_l |\lambda|$. Similarly, applying Lemma \ref{lem:Conductor} to partition $\lambda$, the number of times $\Phi_l(t)$ dividing the denominator $\dim_t S^\lambda$ is ${\rm wt}_l \lambda$. Therefore, $\Phi_l(t)$ divides $\dim_t S^\lambda$ exactly ${\rm wt}_l |\lambda| - {\rm wt}_l \lambda$ many times.

When one specializes to $t = 1$, the result follows from the fact that $\Phi_{p^r}(1) = p$ when $p$ is a prime, and $\Phi_n(1) = 1$ when $n$ is not a prime power.
\end{proof}

\subsection{Relative Cohomology} In this subsection, we follow the constructions in \cite{Ho} and provide discussion of relative cohomology for Hecke algebra.
Let $M$ be a $\He_q(d)$-module, and let $\lambda \vDash d$ be a composition. A {\em relative $\He_q(\lambda)$-projective resolution of $M$} is a resolution of $M$ consisting of relative $\He_q(\lambda)$-projective $\He_q(d)$-modules 
and that splits as resolution of $\He_q(\lambda)$-modules. Among all such resolutions, there exists a {\em minimal resolution}, that is one where there kernels contain no relatively projective summands. The growth rate of the minimal relative $\He_q(\lambda)$-projective resolution of $M$ is called the complexity of $M$, denoted by $c_{d; \lambda}(M):=c_{(\He_{q}(d),\He_q(\lambda))}(M)$.

All relative $\He_q(\lambda)$-projective resolutions are homotopic to each other, and the relative $\Ext$ between two $\He_q(d)$-modules $M$ and $N$ is defined as 
\[ \Ext^\bullet_{(\He_q(d),\He_q(\lambda))} (M, N) := \HH^\bullet(\Hom_{\He_q(d)}(P^\bullet_\lambda, N)), \]
where $P^\bullet_\lambda$ is any relative $\He_q(\lambda)$-projective resolution of $M$.

%Let $P^\bullet_\lambda \twoheadrightarrow M$ be be the minimal projective resolution of $M$. The kernel of $P^{i - 1} \to P^{i - 2}$ is denoted by $\Omega_\lambda^i$. The relative $\Ext$ above has another formula,
%\[ \Ext^\bullet_{\He_q(d); \He_q(\lambda)} (M, N) = \Hom_{\He_q(d)}(\Omega_{\lambda}^\bullet(M), N). \]

Using the same arguments of the proof of self-injectivity of group algebras, one can show that projective modules in $(\He_q(d),\He_q(\lambda))$ are also injective modules. Therefore, all relative $\He_q(\lambda)$-projective resolutions 
with finite length must have length $0$. In particular, $c_{d; \lambda}(M) = 0$ if and only if $M$ is relative $\He_q(\lambda)$-projective. As in the ordinary cohomology case, we showed in Theorem \ref{thm:comparerates} that we can test projectivity of a module $M$ by calculating $\Ext_{\He_q(d)}^\bullet(D, M)$ for all simple modules $D$. The same result holds for relative cohomology as well. More precisely, a $\He_q(d)$-module $M$ is relative $\He_q(\lambda)$-projective if and only if $\Ext^n_{(\He_q(d), \He_q(\lambda))}(D, M) = 0$ for all simple $\He_q(d)$-module $D$ and $n \geq 1$. 

The fact above gives us the following lemma.

\begin{lemma} \label{lem:Demand}
Let $0 \to M_1 \to M_2 \to M_{3} \to 0$ be a short exact sequence of $\He_q(d)$-modules. If any two of $M_1$, $M_2$ and $M_3$ are relative $\He_q(\lambda)$-projective, then so is the third.
\end{lemma}

\begin{proof}
Let $M_i$ and $M_j$ be the two modules that are relative $\He_q(\lambda)$-projective, and let $M_k$ be the third module. The relative complexities of $M_i$ and $M_j$ are zero, so for positive integer $n$ and simple $\He_q(\lambda)$-module $D$, $\Ext^n_{(\He_q(d),\He_q(\lambda))}(D, M_i) = \Ext^n_{(\He_q(d),\He_q(\lambda))}(D, M_j) = 0$. Using the long exact sequence of cohomologies, we get $\Ext^n_{(\He_q(d),\He_q(\lambda))}(D, M_k) = 0$. Therefore, the relative complexity of $M_k$ is zero, and $M_k$ is relative $\He_q(\lambda)$-projective.
\end{proof}

An interesting problem would be to determine whether a suitable support variety theory can be established for the relative cohomology  $(\He_q(d),\He_q(\lambda))$.

\subsection{Vertices for some Specht Modules} We begin this section by discussing blocks and relative projectivity. 

\begin{theorem} \label{thm:Haversack}
Let $\mathbb{B}^\lambda$ be the block of $\He_q(d)$ index by a partition $\lambda \vdash d$. Every module $M$ belongs to $\mathbb{B}^\lambda$ is relative $\He_q(\rho)$-projective for $\rho = (l^{{\rm wt}_l \lambda}, 1^{|{\rm core}_l \lambda|})$
\end{theorem}

\begin{proof}
According to Theorem \ref{youngcomplexity}, every Young module in the block $\mathbb{B}^\lambda$ is relative $\He_q(\rho)$-projective. Young modules have a Specht filtration. By an induction similar to Theorem \ref{thm:comparerates} and Lemma \ref{lem:Demand}, all Specht modules in $\mathbb{B}^\lambda$ are $\He_q(\rho)$-projective. Since Specht modules in $\mathbb{B}^\lambda$ admits filtrations by simple modules in $\mathbb{B}^\lambda$, by an inductive argument similar to Theorem \ref{thm:comparerates} and Lemma \ref{lem:Demand}, all simple modules in $\mathbb{B}^\lambda$ are relative $\He_q(\rho)$-projective. Therefore, by Lemma \ref{lem:Demand}, all modules in $\mathbb{B}^\lambda$ is relative $\He_q(\rho)$-projective.
\end{proof}

By using the previous result on relative projectivity and information about the graded dimension one can obtain information about the vertex for Specht modules. 

\begin{theorem} \label{thm:Shape}
Let $\lambda$ be a partition, and $\rho_a := (l^a, 1^{|\lambda| - a l})$. Assume that $l$ is prime. Then, the vertex of $S^\lambda$ is $\rho_a$ for some $a$ that satisfies
\[ {\rm wt}_l \lambda - \sum_{r \geq 2} {\rm wt}_{l^r} |{\rm core}_{l^r} \lambda| \leq a \leq {\rm wt}_l \lambda. \]
In particular, if $\lambda$ has small $l^r$-core for $r \geq 2$ then $a = {\rm wt}_l \lambda$.
\end{theorem}

\begin{proof}
It is shown in \cite[Section 1.8]{DD} that the vertex of an arbitrary module, particularly $S^\lambda$, is of form $\rho_a$ for some natural number $a$.

Let $\bar{\rho}_a := (1^{l a}, l^{{\rm wt}_l |\lambda| - a}, 1^{{\rm core}_l |\lambda|})$. Since $S^\lambda$ is $\He_q(\rho_a)$-projective,
\[ S^\lambda \downarrow_{\He_q(\bar{\rho}_a)} | \left( \He_q(|\lambda|) \otimes_{\He_q(\rho_a)} S^\lambda \right) \downarrow_{\He_q(\bar{\rho}_a)}. \]
The right hand side of the equation above is a free $\He_q(\bar{\rho}_a)$-module, so the left hand side $S^\lambda \downarrow_{\He_q(\bar{\rho}_a)}$ is a projective $\He_q(\bar{\rho}_a)$-module, and has dimension divisible by $l^{{\rm wt}_l |\lambda| - a}$. 
Note that one can verify that the projective modules in $\He_{q}(l)$ have dimension divisible by $l$ by using their realization as Young modules. 

So, according to Theorem \ref{thm:Sailor},
\[ {\rm wt}_l |\lambda| - a \leq \sum_{r \geq 1} {\rm wt}_{l^r} |{\rm core}_{l^r} \lambda|. \]
Therefore,
\[ a \geq {\rm wt}_l |\lambda| - \sum_{r \geq 1} {\rm wt}_{l^r} |{\rm core}_{l^r} \lambda| = {\rm wt}_l \lambda - \sum_{r \geq 2} {\rm wt}_{l^r} |{\rm core}_{l^r} \lambda|. \]

Theorem \ref{thm:Haversack} insures that the module $S^\lambda$, which is in the block $\mathbb{B}^\lambda$, is relative $\He_q(\rho_{{\rm wt}_l \lambda})$-projective. If $\rho_a$ is the vertex of $S^\lambda$, then 
$\rho_a \vDash \rho_{{\rm wt}_l}$, which implies that $a \leq {\rm wt}_l \lambda$.
\end{proof}

As a consequence of Theorem~\ref{thm:Shape}, one can compute the vertices for Specht modules for partitions whose some of the parts is less than $l^{2}$. 
\begin{cor} \label{cor:Most}
Let $\lambda$ be a partition. Assume that $l$ is prime. If $|\lambda| < l^2$, then the vertex of $S^\lambda$ is $(l^{{\rm wt}_l \lambda}, 1^{|{\rm core}_l \lambda|})$.
\end{cor}

\begin{proof}
For every $r \geq 2$, $|{\rm core}_{l ^ r} \lambda| \leq |\lambda| < l ^ r$, hence $\lambda$ has small $l ^ r$-core. The result follows from Theorem \ref{thm:Shape}.
\end{proof}

\begin{remark}
For the group algebra of symmetric groups $k \Sigma_d$, \cite{Lim} calculated the vertex and the support variety of $S^\lambda$ for many partitions, in particular, when $|\lambda| < p^2$, where $p$ is the characteristic of $k$, 
This can be used in conjunction with the realization of the cohomological support varieties as rank varieties to compute the complexity for Specht modules for symmetric groups in this case. 
\end{remark}

%%%%%%%
%%Section 8
%%%%%%%

\section{Connections to other Hecke algebras for other Weyl groups} 

\subsection{Morita equivalence} We will apply our results for Hecke algebra for other classical groups. 
Our discussion will follow the one given in \cite[Section 6]{BEM}. In the 
case when the root system if of type $B_{n}$ (or $C_{n}$), the Hecke algebra involves two parameters 
$Q$ an $q$. Let $\He_{q,Q}(B_{n})$ denote this Hecke algebra. For type $D_{n}$ the Hecke algebra 
involves one parameter and we will denote this Hecke algebra by $\He_{q}(D_{n})$. 

Consider the following polynomials
\begin{equation} 
f_{n}(Q,q):=\prod_{i=1-n}^{n-1}(Q+q^{i})
\end{equation} 
and 
\begin{equation} 
f_{n}(q):=2\prod_{i=1}^{n-1}(1+q^{i}).
\end{equation} 

We summarize the various Morita equivalence theorems for $\He_{q,Q}(B_{n})$ and 
$\He_{q}(D_{n})$ (cf. \cite[(4.17)]{Dipper-James}, \cite[(3.6) (3.7)]{Pal}). 

\begin{theorem} \label{thm:Morita} Let $\He_{q}(d)$ be the Hecke algebra for the symmetric group $\Sigma_{d}$. 
\begin{itemize} 
\item[(a)] If $f_{n}(Q,q)$ is invertible in $\C$ then $\He_{q,Q}(B_{n})$ is Morita equivalent to 
$$\prod_{j=0}^{n} \He_{q}(j)\otimes \He_{q}(n-j).$$ 
\item[(b)] If $f_{n}(q)$ is invertible in $\C$ and $n$ is odd then $\He_{q}(D_{n})$ is Morita equivalent to 
$$\prod_{j=0}^{(n+1)/2} \He_{q}(j)\otimes \He_{q}(n-j).$$ 
\item[(c)] If $f_{n}(q)$ is invertible in $\C$ and $n$ is even then $\He_{q}(D_{n})$ is Morita equivalent to 
$${\mathcal A}(n/2)\oplus \prod_{j=0}^{(n+1)/2} \He_{q}(j)\otimes \He_{q}(n-j)$$ 
where ${\mathcal A}(n/2)$ is specified in \cite[2.2, 2.4]{Hu}. 
\end{itemize} 
\end{theorem}

\subsection{Support theory for ${\mathcal A}(m)$} Let $n$ be even and set $m=n/2$. The algebra ${\mathcal A}(m)$ as defined in 
\cite{Hu} is an example of a ${\mathbb Z}_{2}$-graded Clifford system (cf. \cite[Section 4]{Hu}). Set  
$B={\mathcal A}(m)$ and $B_{+}$ be the augmentation ideal of $B$. Furthermore, let $A=\He_{q}((m,m))$ be the 
subalgebra in $B$ corresponding to $B_{1}$ (in the Clifford system), and $A_{+}$ be its augmentation ideal. 
Then $B\cdot A_{+}=A_{+}\cdot B$. Now one can consider the quotient $\overline{B}=B//A\cong {\mathbb C}[{\mathbb Z}_{2}]$ 
(the group algebra of the cyclic group of order 2). 

From \cite[5.3 Proposition]{GK}, one can apply the spectral sequence and the fact that $\overline{B}$ is a semisimple algebra to show that 
\begin{equation} 
\operatorname{H}^{\bullet}(B,{\mathbb C})\cong \operatorname{H}^{\bullet}(A,{\mathbb C})^{{\mathbb Z}_{2}}.
\end{equation} 
In fact one can show that $\operatorname{H}^{\bullet}(A,{\mathbb C})$ is an integral extension of $\operatorname{H}^{\bullet}(B,{\mathbb C})$. 
If $M$ is a finite-dimensional $B$-module, we will declare that $V_{B}(M):=V_{A}(M)$ which is defined in Definition~\ref{lem:vsupport}. 

Next we will compare the notion of complexity in $\text{mod}(B)$ versus $\text{mod}(A)$. Since $B$ is a free $A$-module, any projective 
$B$-resolution restricts to a projective $A$-resolution, thus $c_{B}(M)\geq c_{A}(M)$. On the other hand, by \cite[4.4 Corollary]{Hu}, all 
simple $B$-modules are summands of simple $A$-modules induced to $B$. By applying the characterization of complexity given in 
Theorem~\ref{thm:comparerates}(a)(b) and Frobenius reciprocity, one obtains $c_{B}(M)=c_{A}(M)$.

\subsection{} Let ${\mathcal E}_{n}$ be the algebras and $f_{n}:=f_{n}(Q,q)$ (resp. $f_{n}(q)$) be the polynomials as described in Theorem~\ref{thm:Morita} under the Morita equivalence with 
$\He_{q,Q}(B_{n})$ (resp. $\He_{q}(D_{n})$). For notational convenience, set 
\begin{equation} 
\He(\Phi_{n}):= 
\begin{cases} 
\He_{q,Q}(B_{n}), & \text{$\Phi_{n}=B_{n}$}, \\
\He_{q}(D_{n}), & \text{$\Phi_{n}=D_{n}$}. 
\end{cases} 
\end{equation}
Let $F(-):\text{Mod}(\He(\Phi_{n}))\rightarrow \text{Mod}({\mathcal E}_{n})$
be functor that provides the equivalence of categories when $f_{n}$ is invertible. 
Under the equivalence of categories, one can define support varieties for modules over 
$\He(\Phi_{n})$ as follows. Let $M$ be a finite-dimensional module for $\He(\Phi_{n})$. 
Then 
$$V_{\He(\Phi_{n})}(M)=V_{{\mathcal E}_{n}}(F(M)).$$ 
The support varieties for ${\mathcal E}_{n}$ can be obtained by taking the support varieties for Hecke algebras of type $A$. 
We have the following theorem that extends Corollary~\ref{c:com-support}. 

\begin{theorem} Let $M$ be a finite-dimensional module for $\He(\Phi_n)$ with $f_{n}$ invertible. Then 
\begin{itemize} 
\item[(a)] $c_{\He(\Phi_{n})}(M)=\dim V_{\He(\Phi_n)}(M)$. 
\item[(b)] $c_{\He(\Phi_{n})}(M)\leq c_{\He(\Phi_{n})}(\C)=\lfloor \frac{n}{l} \rfloor$. 
\end{itemize} 
\end{theorem} 

\begin{proof} (a) Let $S=\oplus_{i} S_{i}$ be the direct sum of simple $\He(\phi_{n})$-modules. Using the Morita equivalence, $F(S)$ is 
the direct sum of simple ${\mathcal E}_{n}$-modules. Furthermore, by using our results for the Hecke algebra for type $A$, 
\begin{eqnarray*} 
c_{\He(\Phi_{n})}(M)&=& r(\Ext^{\bullet}_{\He(\Phi_{n})}(S,M)) \\
&=& r(\Ext^{\bullet}_{{\mathcal E}_{n}}(F(S),F(M))) \\
%&=& c_{{\mathcal E}_{n}}(F(M)) \\
&=& \dim V_{{\mathcal E}_{n}}(F(M)) \\
&=& \dim V_{\He(\Phi_{n})}(M).
\end{eqnarray*} 

(b) One has that 
\[ c_{\He(\Phi_{n})}(M)=\dim V_{{\mathcal E}_{n}}(F(M))\leq \left\lfloor \frac{n}{l} \right\rfloor. \]
Let $L$ be the irreducible ${\mathcal E}_{n}$-module such that $F(\C)=T$. Under the categorical equivalence, the trivial module $\C$ goes to the simple 
${\mathcal E}_{n}$-module labelled by the partition $((n),\varnothing)$. 
The statement now follows because 
\[ c_{\He(\Phi_{n})}(\C)=\dim V_{{\mathcal E}_{n}}(T)=\left\lfloor \frac{n}{l} \right\rfloor. \qedhere \]
\end{proof} 	

By using the Morita equivalence one can prove analogs of Theorem~\ref{thm:location} for the blocks of $\He(\Phi_{n})$ and obtain the location of their support varieties 
for various modules. One can pose an interesting question if one can (i) extend the support variety theory for Hecke algebra of types $B_{n}$ and $D_{n}$ to even roots of unity, and (ii) if a theory of 
support varieties can be developed in for Hecke algebras of other Coxeter groups. 

\let\section=\oldsection

\end{document}